\documentclass{elsarticle}
\usepackage{amsthm}
\usepackage{amsmath}
\usepackage{amssymb}
\usepackage{enumerate}
\usepackage{color}
\usepackage[labelsep=quad,indention=10pt]{caption}
\usepackage[labelfont=bf,list=true]{subcaption}

\usepackage{tikz}

\makeatletter
\def\ps@pprintTitle{%
  \let\@oddhead\@empty
  \let\@evenhead\@empty
  \def\@oddfoot{\reset@font\hfil\thepage\hfil}
  \let\@evenfoot\@oddfoot
}
\makeatother

\newtheorem{theorem}{Theorem}[section]

\newtheorem{lemma}[theorem]{Lemma}

\newtheorem{algorithm}[theorem]{Algorithm}
\theoremstyle{definition}
\newtheorem{definition}[theorem]{Definition}

\newcommand\old[1]{}

\begin{document}
\begin{frontmatter}
\title{Half-regular factorizations of the complete bipartite graph\tnoteref{ack}}
\author[bsm]{Mark Aksen \fnref{Mark}}
\author[bsm,renyi,sztaki]{Istvan Miklos\fnref{Istvan}}
\author[bsm]{Kathleen Zhou\fnref{Kate}}

\address[bsm]{Budapest Semesters in Mathematics, Bethlen G\'abor t\'er 2, Budapest, 1071 Hungary\\
{\tt email:} maksen@princeton.edu, zhou.kat@gmail.com}

\address[renyi]{Alfr\'ed R{\'e}nyi Institute, Re\'altanoda u 13-15, Budapest, 1053 Hungary\\
{\tt email:} miklos.istvan@renyi.mta.hu}

\address[sztaki]{Institute for Computer Science and Control, L\'agym\'anyosi \'ut 11, Budapest, 1111 Hungary}
\fntext[miklos]{Partly supported by Hungarian NSF, under contract K116769.}
\begin{abstract}
We consider a bipartite version of the color degree matrix problem.
A bipartite graph $G(U,V,E)$ is half-regular if all vertices in $U$ have the same degree.
We give necessary and sufficient conditions for a bipartite degree matrix (also known as demand matrix) to be the color degree matrix of an edge-disjoint union of half-regular graphs. We also give necessary and sufficient perturbations to transform realizations of a half-regular degree matrix into each other. Based on these perturbations, a Markov chain Monte Carlo method is designed in which the inverse of the acceptance ratios are polynomial bounded.

Realizations of a half-regular degree matrix are generalizations of Latin squares, and they also appear in applied neuroscience. 
\end{abstract}
\end{frontmatter}

\section{Introduction}\label{sec:intro}
Consider an edge weighted directed graph as a model of a neural network. Such network can be build up using real life measurements (see for example, \cite{markovetal2014}), and neuroscientists are interested in comparing this network with random networks. If the edges were not weighted, the typical approach would be to generate random graphs with prescribed in and out degrees. This topic has a tremendous literature, see for example, \cite{lamar2009,greenhill2011,kimetal2012,delgenio2014,ekms2015}, just to mention a few.

In case of weights are introduced, one might want to generate a random graph that keeps not only the sum of the weights but also individual weights. The weights can be transformed into colors, and then we are looking for an edge colored graph with prescribed in and out degrees for each color. This problem is known as finding edge packing \cite{buschetal2012}, edge disjoint realizations \cite{gmt2011}, or degree constrained edge-partitioning \cite{bentzetal2009}. The problem has also applications in discrete tomography \cite{bentzetal2009,dgm2012}. Above finding one edge colored graph with given constraints, it is also an important problem how to generate a "typical" solution, as we can see in the before mentioned neuroscientific problem.

Unfortunately, the general edge packing problem is NP-complete even if the number of colors is 2 \cite{cd1998,ggd1999,dgm2012}. Although the general problem is NP-complete, special cases are tractable. Such special cases include the case when the number of colors is 2 and the subgraph with one of the colors is almost regular \cite{kundu1974,kw1973,chen1988}. Another tractable case is when the graph is bipartite, the number of colors is 2, and there exist constants $k_1$ and $k_2$ such that for each vertex, the total number of edges in one of the vertex class is $k_1-1$, $k_1$ or $k_1+1$ and in the other vertex class is $k_2-1$, $k_2$ or $k_2+1$ \cite{gmt2011}. When the number of colors is unlimited, tractable solutions exist if the graphs are forests for each color \cite{bentzetal2009} or the graphs are forests for each color except one of the colors might contain one cycle \cite{hmcd2015}.

A special case is when the bipartite graph is the union of $n$ 1-factors. Such edge packings are simply the Latin squares. For any $n$, Latin squares exist, and even a Markov chain is known that explores the space of Latin squares for a fixed $n$ \cite{jm1996}. This Markov chain is conjectured to be rapidly mixing, that is, computationally efficient to sample random Latin squares. The conjecture has neither proved nor disproved in the last twenty years.

In this paper, we consider another tractable case with unlimited number of colors. We require that the subgraphs for each color be half-regular, that is, the degrees are constant on one of the vertex class. We further require that this regular vertex class be the same for all colors. We give sufficient and necessary conditions when this edge packing problem has a solution. We also give necessary and sufficient perturbations to transform solutions into each other. A Markov chain Monte Carlo method has been given using these perturbations, and we give a proof that the inverse of the acceptance ratio is polynomial bounded.

\section{Preliminaries}\label{sec:pre}

In this paper, we are going to work with realizations of half-regular degree matrices, defined below.

\begin{definition}
A \emph{bipartite degree sequence} $D =\left(( d_1, d_2, \ldots d_n), (f_1, f_2, \ldots f_m)\right)$ is a pair of sequences of non-negative integers. A bipartite degree sequence is \emph{graphical} if there exists a simple bipartite graph $G$ whose degrees correspond exactly to $D$. We say that $G$ is a \emph{realization} of $D$.
\end{definition}

\begin{definition}
A \emph{bipartite degree matrix} $\mathcal{M} = (D,F)$  is a pair of $k \times n$ and $k \times m$ matrices of non-negative integers. A bipartite degree matrix is graphical if there exists an edge colored simple bipartite graph $G(U,V,E)$ such that for all color $c_i$ and for all $u_j \in U$, the number of edges of $u_j$ with color $c_i$ is $d_{i,j}$ and for all $v_l \in V$, the number of edges of $v_l$ with color $c_i$ is $f_{i,l}$. Such graph is called a \emph{realization} of $\mathcal{M}$. 
A bipartite degree matrix is \emph{half-regular} if for all $i,j,l$, $d_{i,j} = d_{i,l}$.

The rows of $\mathcal{M}$ are bipartite degree sequences that are also called \emph{factor}s and the edge colored realization of $\mathcal{M}$ is also called an \emph{$\mathcal{M}$-factorization}.
\end{definition}

We consider two problems. One is the existence problem which asks the question if there is a realization of a half-regular degree matrix. The other is the sampling problem, which considers the set of all realizations of a half-regular degree matrix and asks the question how to sample uniformly a realization from this set. Markov Chain Monte Carlo (MCMC) methods are generally applicable for such problems, and we are also going to introduce an MCMC for sampling realizations of half-regular degree sequences. Below we introduce the main definitions.

\begin{definition}
A \emph{discrete time, finite Markov chain} is a random process that undergoes transitions from one state to another in a finite state space. The process has the Markov property which means that the distribution of the next state depends only on the current state. The transition probabilities can be described with a transition matrix $\mathbf{T} = \{t_{i,j}\}$, where $t_{i,j} := P(x_i|x_j)$, namely, the conditional probability that the next state is $x_i$ given that the current state is $x_j$. When the state space is a large set of combinatorial objects, the transition probabilities are
not given explicitly, rather, a random algorithm is given that generates a random $x_i$ by perturbing the current state $x_j$. Such algorithm is called \emph{transition kernel}. 
\end{definition}

When the process starts in a state $x_i$, after $t$ number of steps, it will be in a random state
with distribution $\mathbf{T}^t \mathbf{1}_i$, where $\mathbf{1}_i$ is the column vector containing all $0$'s except for coordinate $i$, which is $1$. The Markov Chain Monte Carlo method is the way to tailor the transition probabilities such that the limit distribution $\lim_{t \rightarrow \infty}\mathbf{T}^t \mathbf{1}_i$ be a prescribed distribution. To be able to do this, necessary conditions are that the Markov chain be irreducible, aperiodic and the transitions be reversible defined below.

\begin{definition}
Given a discrete time, finite Markov chain on the state space $\mathcal{I}$, the \emph{Markov graph} of the Markov chain is a directed graph $G(V,E)$, with vertex set $\mathcal{I}$ and there is an edge from $v_i$ to $v_j$ iff $P(v_j|v_i) \ne 0$. A Markov chain is \emph{irreducible} iff its Markov graph is strongly connected. When the transition kernel of an irreducible Markov chain generates a class of perturbations, we also say that this class of perturbations is irreducible on the state space.

A Markov chain is \emph{aperiodic} if the largest common divisor of cycle lengths of its Markov graph is 1. This automatically holds, if there is a loop in the Markov chain, that is, a state $x$ exists for which $P(x|x) \ne 0$.
\end{definition}

\begin{definition}
The transition kernel of a Markov chain is reversible if for all $x_i, x_j$, $P(x_i|x_j) \ne 0 \Leftrightarrow P(x_j|x_i) \ne 0$.
\end{definition}

The strength of the theory of Markov Chain Monte Carlo is that any irreducible, aperiodic Markov chain with reversible transition kernel can be tailored into a Markov chain converging to a prescribed distribution as stated below.

\begin{theorem} \cite{metropolisetal1953,hastings1970}
Let an irreducible, aperiodic Markov chain be given with reversible transition kernel $T$ over the finite state space $\mathcal{I}$. Let $\pi$ be a distribution over $\mathcal{I}$, for which $\forall x \in \mathcal{I}$, $\pi(x) \ne 0$. Then the following algorithm, called the Metropolis-Hastings algorithm, also defines a Markov chain that converges to $\pi$, namely, its transition kernel $\mathbf{T'}$ satisfies $\lim_{t\rightarrow \infty} \mathbf{T'}^t \mathbf{1}_i = \pi$ for all indices $i$.
\begin{enumerate}
\item Draw a random $y$ following the distribution $T(\cdot|x_t)$ where $x_t$ is the current state of the Markov chain after $t$ steps.
\item Draw a random $u$ following the uniform distribution over $[0,1]$. The next state, $x_{t+1}$, will be $y$ if 
$$u \le \min\left\{1, \frac{\pi(y)T(x_t|y)}{\pi(x_t)T(y|x_t)}\right\}$$
and $x_t$ otherwise.
\end{enumerate}
\end{theorem}

The ratio in the second step of the algorithm is called the Metropolis-Hastings ratio. It is simplified
to $\frac{T(x_t|y)}{T(y|x_t)}$ when the target distribution $\pi$ is the uniform one. When this ratio is small, the Markov chain defined by the algorithm tends to remain in the same state for a long time, thus increasing the mixing time, namely, the number of steps necessary to get close to the target distribution.

In this paper, we introduce perturbations that are irreducible on the realizations of a half-regular degree matrix. Using these perturbations, we design a reversible transition kernel, for which 
\begin{equation}
\frac{T(x_t|y)}{T(y|x_t)} \ge \frac{2}{m^5} \,\,\,\forall x_t, y
\label{eq:highratio}
\end{equation}
where $m$ is the number of vertices in vertex class $V$.

\section{The existence problem}

\begin{definition}
Let $G(V,E)$ be a multigraph, and for each $u,v \in V$, let $f(u,v)$ be
\begin{equation}
f(u,v) := \begin{cases}
0 & \mathrm{if} \ (u,v) \notin E \\
m(u,v) - 1 & \mathrm{otherwise}
\end{cases}
\end{equation}
where $m(u,v)$ denotes the multiplicity of edge $(u,v)$. The \emph{exceed number} of graph $G$ is defined as
\begin{equation}
ex(G) := \sum_{u,v \in V} f(u,v)
\end{equation}
\end{definition}

Clearly, the exceed number is $0$ iff $G$ is a simple graph.

\begin{theorem}\label{the:regular-factorization}
Let
\begin{eqnarray}
\mathcal{M} = \{\{d_{1,1} = d_{1,2} = \ldots = d_{1,n}\}, && \{f_{1,1}, f_{1,2}, \ldots f_{1,m}\} \nonumber\\
\{d_{2,1} = d_{2,2} = \ldots = d_{2,n}\}, && \{f_{2,1}, f_{2,2}, \ldots f_{2,m}\} \nonumber \\
\vdots & & \nonumber \\ 
\{d_{k,1} = d_{k,2} = \ldots = d_{k,n}\}, && \{f_{k,1}, f_{k,2}, \ldots f_{k,m}\}\}
\end{eqnarray}
be a half-regular bipartite degree matrix. The bipartite complete graph $K_{n,m}$ has an $\mathcal{M}$-factorization iff
\begin{enumerate}
\item $\forall i$, $n d_{i,1} = \sum_{j=1}^m f_{i,j}$
\item $\sum_{i=1}^k d_{i,1} = m$
\item $\forall j$, $\sum_{i=1}^k f_{i,j} =n$.
\end{enumerate}
\end{theorem}
\begin{proof}
$\Rightarrow$ If $K_{n,m}$ has an $\mathcal{M}$-factorization, then clearly, all factors are graphical and the degrees sum to the degrees of the complete bipartite graph. Conditions $2$ and $3$ explicitly state that the sum of the degrees in the factors sum up to the degree of the complete bipartite graph. Condition $1$ states that in each factor, the sums of the degrees in the two vertex classes are the same, which is a necessary condition for a graphical degree sequence.

$\Leftarrow$ Conditions $2$ and $3$ also say implicitly that for each $i$, $d_{i,1} \le m$ and for each $i,j$, $f_{i,j} \le n$. Together with condition $1$, they are sufficient conditions that a half-regular bipartite degree sequence be graphical. Namely, the theorem says if all factors are graphical and the sums of the degrees are the degrees of the complete bipartite graph, then $K_{n,m}$ has such a factorization. 

For each $i$, let $G_i$ be a realization of the degree sequence $(d_{i,1}, \ldots d_{i,n}), (f_{i,1}, \ldots f_{i,m})$. Let $G = \cup_{i=1}^k G_i$. Color the edges of $G$ such that an edge is colored by color $c_i$ if it comes from the realization $G_i$. $G$ might be a multigraph, however, for each color $c_i$ and each pair of vertices $u,v$, there can be at most one edge between $u$ and $v$ with color $c_i$. If $ex(G) = 0$, then $G$ is a simple graph, and due to the conditions, it is $K_{n,m}$, thus we found an $\mathcal{M}$-factorization of $K_{n,m}$.

Assume that $ex(G) > 0$. Then there is a pair of vertices $u,v$ such that there are more than one edge between $u$ and $v$. Fix one such pair $(u,v)$, and let $V_0$ denote the set $\{v\}$. Since the degree of $v$ in $G$ is $n$, there must be a $u'$ such that there is no edge between $u'$ and $v$. Let $C_0$ denote the set of colors that appear as edge colors between $u$ and $v$. Since for each color $c_i$, $u$ and $u'$ has the same number of edges with color $c_i$, there must be at least one vertex $v'$ such that there are more edges with color from $C_0$ between $u'$ and $v'$ than the edges with color from $C_0$ between $u$ and $v'$. Let $V'$ denote the set of vertices for which this condition holds. There are three possibilities (see also Figure~\ref{fig:firstiteration}):
\begin{enumerate}
\item There is a vertex $v' \in V'$ such that there are more than one edge between $u'$ and $v'$.
\item There is a vertex $v' \in V'$ such that there is only a single edge between $u'$ and $v'$, however, there is no edge between $u$ and $v'$.
\item For each vertex $v' \in V'$, there is only a single edge between $u'$ and $v'$ and there is at least one edge between $u$ and $v'$.
\end{enumerate}
\begin{figure}\centering
\includegraphics[trim=6cm 2cm 8cm 10cm, angle=0, width=2.9in]{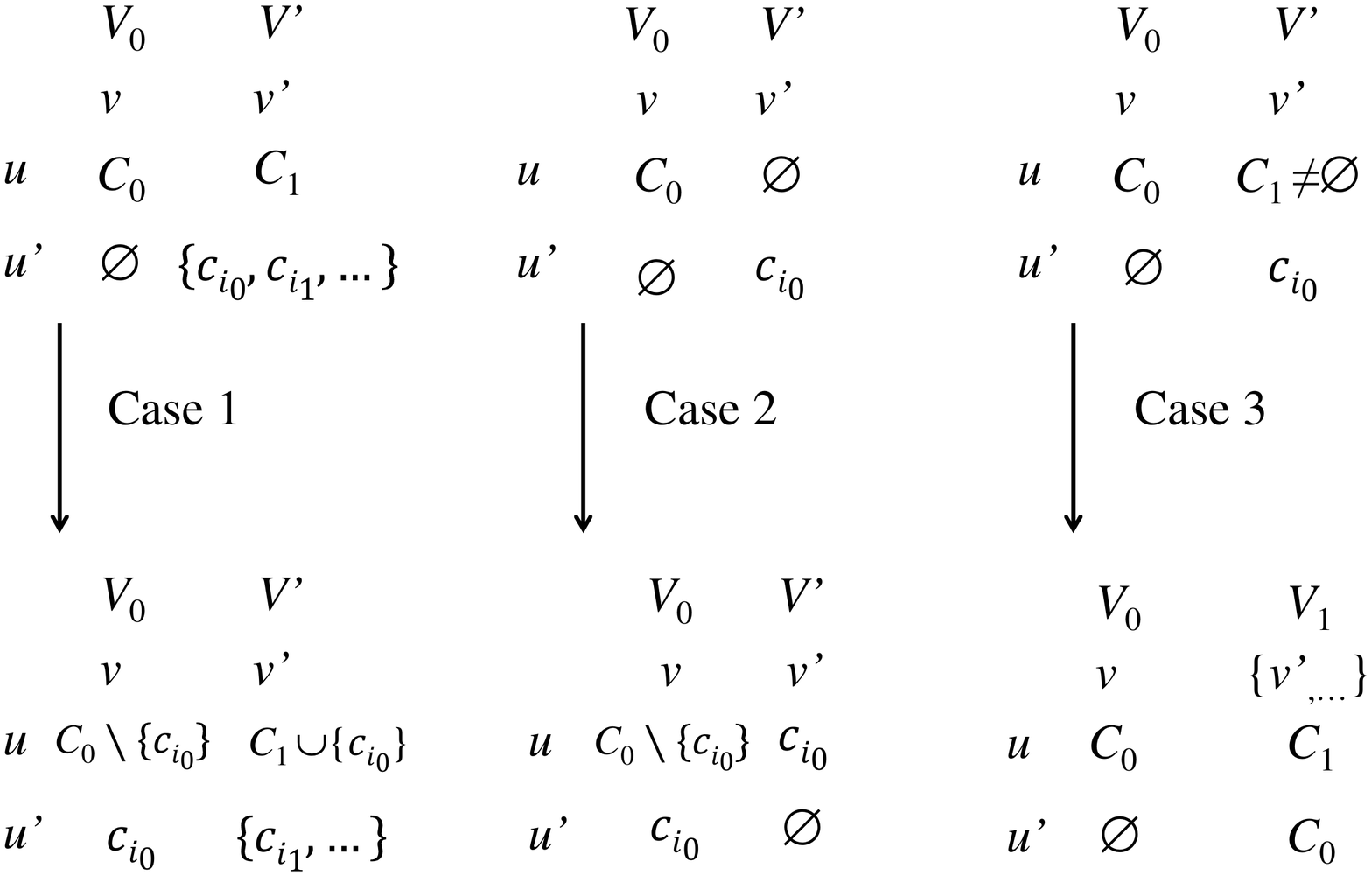}
\caption{The three possible cases when there are more than one edge between vertices $u$ and $v$. See text for details.}\label{fig:firstiteration}
\end{figure}
If case $1$ holds, then let $c_{i_0}$ be a color in $C_0$ such that there is a $c_{i_0}$ colored edge between $u'$ and $v'$ and there is no $c_{i_0}$ colored edge between $u$ and $v'$. Remove that edge between $u$ and $v$ and also between $u'$ and $v'$ and add a $c_{i_0}$ colored edge between $u'$ and $v$ and also between $u$ and $v'$. Note that $G_{i_0}$, the graph with color $c_{i_0}$ remains a simple graph with the prescribed degree sequence.
$f(u,v)$ decreases by $1$. $f(u',v)$ remain $0$, since the edge with color $c_{i_0}$ is the only edge between them. $f(u',v')$ decreases by $1$. $f(u,v')$ might increase by $1$ if there was already at least one edge between them before the change, otherwise it remains the same. Altogether, $ex(G)$ is decreased at least by $1$.

If case $2$ holds, then do the same edge changing than in case $1$. This causes a decrease in $f(u,v)$ by one, and no change in $f(u',v)$, $f(u,v')$ and $f(u',v')$. Altogether, $ex(G)$ decreases by $1$.

In case $3$, select a subset $V_1$ from $V'$ such that the colors of edges between $u'$ and the vertices in $V_1$ are exactly the set $C_0$. 
This is doable, since for each vertex $v' \in V'$, there is only one edge between  $u'$ and $v'$, furthermore, the union of colors of edges between $u'$ and $V'$ must contain a subset $C_0$.
Let $C_1$ denote the (possibly multi)set of colors that appear as edge colors between vertex $u$ and the set of vertices $V_1$. The multiset of colors between $u$ and $V_0 \cup V_1$ is $C_0 \uplus C_1$, where $\uplus$ denote the multiset union, while the set of colors between $u'$ and $V_0 \cup V_1$ is $C_0$. Therefore, there is at least one vertex $v' \notin V_0 \cup V_1$ such that the number of edges with colors from $C_1$ between $u'$ and $v'$ is more than the number of edges with color from $C_1$ between $u$ and $v'$. Let $V'$ denote the set of vertices with this property. If case $3$ holds, then we can select a subest $V_2$ from $V'$ such that the multiset of colors between vertex $u'$ and $V_2$ is exactly $C_1$. Then the multiset of colors between $u$ and $V_0 \cup V_1 \cup V_2$ is $C_0 \uplus C_1 \uplus C_2$, while the multiset of colors between $u'$ and $V_0 \cup V_1 \cup V_2$ is $C_0 \uplus C_1$.

Therefore, while case $3$ holds, we can select the subset of vertices $v' \notin \cup_{i=0}^{j-1} V_i$ such that the number of edges with colors from $C_{j-1}$ between $u'$  and $v'$ is more than the number of edges with colors from $C_{j-1}$ between $u$ and $v'$, and from this set, we can select an appropriate subset $V_{j}$. Since there are finite number of vertices in $V_j$, for some $j$, case $1$ or $2$ must hold (see also Figure~\ref{fig:iteration}). Let $v'$ be the vertex for which case $1$ or $2$ holds, and let $c_{i_{j-1}}$ be the color such that there is an edge with color $c_{i_{j-1}}$ between $v'$ and $u'$ and there is no such edge between $u$ and $v'$. Remove that edge between $u'$ and $v'$ and add between $u$ and $v'$. Let $v_{j-1} \in V_{j-1}$ be a vertex such that there is an edge between $u$ and $v_{j-1}$ with color $c_{i_{j-1}}$. Remove that edge and add between $u'$ and $v_{j-1}$. Let the color between $u'$ and $v_{j-1}$ be $c_{i_{j-2}}$. Remove it and add between $u$ and $v_{j-1}$. (Note that due to the definition of $V'$ and due to case $3$ held in the $j-1$st iteration, before the change, there was no $c_{i_{j-2}}$ colored edge between $u$ and $v_{j-1}$, thus after the change, all edges between $u$ and $v_{j-1}$ have different colors.) Iterate this process; in the $l$th iteration, let $v_{j-l}$ be a vertex in $V_{j-l}$ such that there is an edge with color $c_{i_{j-l}}$ between $u$ and $v_{j-l}$. Remove that edge and add between $u'$ and $v_{j-l}$, remove the edge with color $c_{i_{j-l-1}}$ between $u'$ and $v_{j-l}$ and add between $u$ and $v_{j-l}$. Finally, remove the edge with color $c_{i_0}$ between $u$ and $v$ and add it between $u'$ and $v$. Let this new graph be $G'$. Then in $G'$, $f(u,v)$ decreases by 1, $f(u',v)$ remains the same. For each $i=1,\ldots j-1$, neither $f(u,v_i)$ nor $f(u',v_i)$ is changed since the total number of edges between them is not changed. Finally, $f(u,v')+f(u',v')$ is not increased. Altogether, $ex(G')$ is at least $1$ less than $ex(G)$. 
On Figure~\ref{fig:iteration}, it is easy to verify that all vertices participating in the modification of $G$, the same colored edges were removed and added. Therefore, $G'$ is still the union of realizations of the prescribed degree sequences.
\begin{figure}\centering
\includegraphics[trim= 8cm 2cm 10cm 7cm,angle=0, width=2.3in]{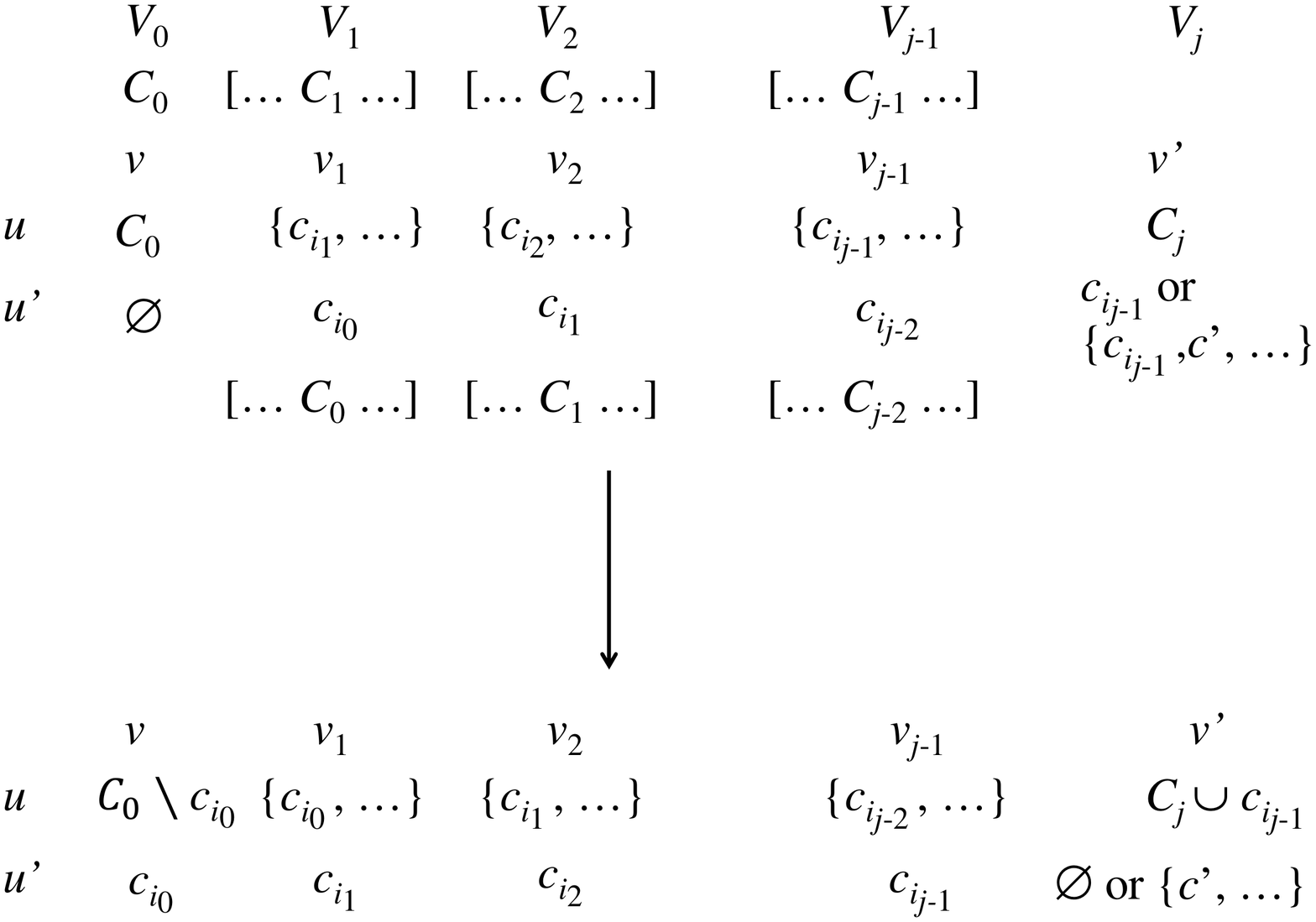}
\caption{The case when after $j$ iterations, either Case 1 or Case 2 holds. 
$[\ldots C_l \ldots]$ indicates that the (possibly multi)set of colors between $u$ and $V_l$ is $C_l$ 
and the (possibly multi)set of colors between $u'$ and $V_l$ is $C_{l-1}$, for all $l= 1, \ldots j-1$. From the set of vertices $V_l$, a vertex $v_l$ is selected such that there is one edge between $u'$ and $v_l$ with color $c_{i_{l-1}}$ and there is at least one edge between $u$ and $v_l$ including an edge with color $c_{i_{l}}$.
See text for details.}\label{fig:iteration}
\end{figure}

Since the exceed number is a non-negative finite integer, after a finite number of steps, the exceed number will be $0$, thus we get $K_{n,m}$ as an $\mathcal{M}$-factorization.
\end{proof}

It is clear that one of the colors might encode "non-edge", namely, we can delete those edges and get a realization of a half-regular degree matrix obtained by deleting one row from $\mathcal{M}$. Therefore, the following theorem also holds.

\begin{theorem}\label{theo:extended halfregular}
Let
\begin{eqnarray}
\mathcal{M} = \{\{d_{1,1} = d_{1,2} = \ldots = d_{1,n}\}, && \{f_{1,1}, f_{1,2}, \ldots f_{1,m}\} \nonumber\\
\{d_{2,1} = d_{2,2} = \ldots = d_{2,n}\}, && \{f_{2,1}, f_{2,2}, \ldots f_{2,m}\} \nonumber \\
\vdots & & \nonumber \\ 
\{d_{k,1} = d_{k,2} = \ldots = d_{k,n}\}, && \{f_{k,1}, f_{k,2}, \ldots f_{k,m}\}\}
\end{eqnarray}
be a half-regular bipartite degree matrix. Then $\mathcal{M}$ has a realization iff
\begin{enumerate}
\item $\forall i$, $n d_{i,1} = \sum_{j=1}^m f_{i,j}$
\item $\sum_{i=1}^k d_{i,1} \le m$
\item $\forall j$, $\sum_{i=1}^k f_{i,j} \le n$.
\end{enumerate}
\end{theorem}
\begin{proof}
It is clear that equality in condition 2 holds iff equality in condition 3 holds for all $j$. In case of equality, we get back Theorem~\ref{the:regular-factorization}. In case of inequality, let 
$$d_{k+1,1} = d_{k+1,2} = \ldots d_{k+1,n} = m - \sum_{i=1}^k d_{i,1}$$
and let
$$ f_{k+1,i} = n - \sum_{j=1}^k f_{j,i}\,\,\,\, \forall i = 1, 2, \ldots m.$$
Extend $\mathcal{M}$ with this $k+1$st row, and this matrix will satisfy the conditions of Theorem~\ref{the:regular-factorization}. Find a realization of this extended matrix, and delete the edges with the $k+1$st color, thus obtain a realization of $\mathcal{M}$.
\end{proof}

In the following section, we consider the solution space of realizations of half-regular degree matrices. Since there is no difference of realizations of half-regular degree matrices and $\mathcal{M}$-factorizations of the complete bipartite graph $K_{n,m}$, we will consider this later, namely, we consider non-edges as $k+1$st colors when non-edges exist.

\section{The connectivity problem}

In this section, we give necessary and sufficient perturbations to transform any realization of a half-regular bipartite degree matrix, $\mathcal{M}$, into another realization of $\mathcal{M}$. 
First, we extend the space of graphs on which perturbations are applied. We show how to transform a realization of $\mathcal{M}$ into another realization in this extended space, then
we show how to transform realizations into each other remaining in the space of realizations.
The concept is very similar to the concept applied in paper \cite{jm1996}, but different
perturbations are necessary to temporarily extend the space of graphs. First, we introduce the
necessary definition, the $(+c_1-c_2)$ deficiency. We will extend the space of graphs that have 
at most $3$ vertices with deficiency. We show how to transform $G_1$, a realization of 
$\mathcal{M}$ into another realization $G_2$, via graphs having at most $3$ vertices with 
deficiency. In doing so, we first arrange the colored edges of a vertex $u_0 \in U$ that they 
agree with realization $G_2$. Then we fix these edges, and reduce the problem to a similar, 
smaller problem. Note that $U$ is the regular class, and if we fix (technically: remove) the vertex $u_0$ together
with its edges, the remaining graph is still half-regular. Finally, we prove how to transform 
realizations of  $\mathcal{M}$ into each other remaining in the space of realizations of 
$\mathcal{M}$.

First, we define deficiency.

\begin{definition}
Let $\mathcal{M}$ be a bipartite degree matrix, and let $G$ be an edge colored bipartite graph. We say that a vertex $u_j$ in $G$ has a $(+c_{i_1} - c_{i_2})$-deficiency w.r.t. $\mathcal{M}$ if the number of its $c_{i_1}$ colored edges is $d_{i_1,j}+1$, the number of its $c_{i_2}$ colored edges is $d_{i_2,j}-1$ and for all other $i \neq i_1,i_2$, the number of $c_i$ colored edges of $u_j$ is $d_{i,j}$.
\end{definition}

Then we define an auxiliary graph that we use several times.

\begin{definition}\label{def:auxiliaryK}
Let $G(U,V,E)$ be an edge colored bipartite graph in which the edges are colored with colors $c_1, c_2, \ldots c_k$, and let $u$ and $u'$ be two vertices in $U$. Then the directed, edge labeled multigraph $K(G,u,u')$ is defined in the following way. The vertices of $K$ are the colors $c_1, c_2, \ldots c_k$, and for each $v \in V$, there is an edge going from $c_i$ to $c_j$, where $c_i$ is the color of the edge between $u$ and $v$ and $c_j$ is the color of the edge between $u'$ and $v$. Such edge is labeled with $v$.
\end{definition}

The next two lemmas show how to handle graphs with a small amount of deficiency.

\begin{lemma}\label{lem:broken-permutation}
Let $G$ be such an edge colored simple graph of $K_{n,m}$ which is almost a factorization of a half-regular bipartite degree matrix $\mathcal{M} = (D,F)$ in the following sense:
\begin{enumerate}
\item For each color $c_i$ and vertex $v_j$, the number of edges with color $c_i$ is $f_{i,j}$.
\item For each color $c_i$ and vertex $u_j$ the number of edges with color $c_i$ is $d_{i,j}$ except for two colors $c_{i_1}$ and $c_{i_2}$ and two vertices $u_{j_1}$ and $u_{j_2}$, where $u_{j_1}$ has a $(+c_{i_1} - c_{i_2})$-deficiency and $u_{j_2}$ has a $(+c_{i_2} - c_{i_1})$-deficiency. 
\end{enumerate}
Then there exists a perturbation of $G$ that affects only edges of $u_{j_1}$ and $u_{j_2}$ and transforms $G$ into a realization of $\mathcal{M}$.
\end{lemma}
\begin{proof}
Consider $K(G,u_{j_1},u_{j_2})$. For each vertex $c_i$, $i \neq i_1, i_2$ of the graph $K(G,u_{j_1},u_{j_2})$, the number of incoming and outgoing edges are the same, while $c_{i_1}$ has two more outgoing edges than incoming and $c_{i_2}$ has two more incoming edges than outgoing. Therefore, there is a trail from $c_{i_1}$ to $c_{i_2}$ due to the pigeonhole principle. For each edge $e_h$ labeled by $v_h$ along the trail, swap the corresponding edges in $G$, namely, the edge between $u_{j_1}$ and $v_h$ and the edge between $u_{j_2}$ and $v_h$. This transforms $G$ into a realization of $\mathcal{M}$. Indeed, $u_{j_1}$ will have one less edge with color $c_{i_1}$ and one more edge with color $c_{i_2}$ while the effect on $u_{j_2}$ is the opposite. The number of edges with other colors are not affected, since in the trail, the number of incoming and outgoing edges are the same for all colors not $c_{i_1}$ and not $c_{i_2}$, and swapping the edges in $G$ is equivalent with inverting the direction of the corresponding edges in $K(G,u_{j_1},u_{j_2})$.
\end{proof}

\begin{lemma}\label{lem:broken-permutation-ii}
Let $G$ be such an edge colored simple graph of $K_{n,m}$ which is almost a factorization of a half-regular bipartite degree matrix $\mathcal{M} = (D,F)$ in the following sense
\begin{enumerate}
\item For each color $c_i$ and vertex $v_j$, the number of edges with color $c_i$ is $f_{i,j}$.
\item For each color $c_i$ and vertex $u_j$, the number of edges with color $c_i$ is $d_{i,j}$ except for three colors $c_{i_1}, c_{i_2}$ and $c_{i_3}$ and three vertices $u_{j_1}$, $u_{j_2}$ and $u_{j_3}$, for which $u_{j_1}$ has $(+c_{i_1} - c_{i_2})$-deficiency, $u_{j_2}$ has $(+c_{i_2} - c_{i_3})$-deficiency and $u_{j_3}$ has $(+c_{i_3} - c_{i_1})$-deficiency.
\end{enumerate}
Then there exists a perturbation of $G$ that affects only edges of $u_{i_2}$ and $u_{i_3}$ and transforms $G$ into an edge colored simple graph of $K_{n,m}$ satisfying conditions 1 and 2 in Lemma~\ref{lem:broken-permutation}. 
\end{lemma}

\begin{proof}
Consider the graph $K(G,u_{j_3},u_{j_2})$.
For each vertex $c_i$, $i \neq i_1, i_2, i_3$ of the graph $K(G,u_{j_3},u_{j_2})$, the number of incoming and outgoing edges are the same, while $c_{i_3}$ has two more outgoing edges than incoming and $c_{i_1}$ and $c_{i_2}$ has one-one more incoming edges than outgoing. Therefore there is a trail from vertex $c_{i_3}$ to $c_{i_2}$ (and also to $c_{i_1}$) due to the pigeonhole principle.
Take a trail from $c_{i_3}$ to $c_{i_1}$, and for each edge along the trail, swap the corresponding edges in $G$, namely, the edge between $u_{j_2}$ and $v_h$ and the edge between $u_{j_3}$ and $v_h$. This transforms $G$ into a graph satisfying conditions 1 and 2 in Lemma~\ref{lem:broken-permutation}. Indeed, $u_3$ will have one less edge with color $c_{i_3}$ and one more edge with color $c_{i_1}$, namely, the transformation cancels its deficiency. Vertex $u_2$ will have one more edge with color $c_{i_3}$ and one less edge with color $c_{i_1}$, therefore its $(+c_{i_2}-c_{i_3})$ deficiency becomes a $(+c_{i_2}-c_{i_1})$ deficiency. The number of edges with other colors are not affected, since in the trail, the number of incoming and outgoing edges are the same for all colors not $c_{i_1}$ and not $c_{i_3}$, and swapping the edges in $G$ is equivalent with inverting the direction of the corresponding edges in $K(G,u_{j_3},u_{j_2})$.
\end{proof}


The following lemma is the key lemma in transforming a realization into another realization. As we
mentioned above, the strategy is to transform a realization $G_1$ into an intermediate realization
$H$, such that the colors of edges of a vertex $u_0$ in the regular class agrees with the colors
of the edges of $u_0$ in the target realization $G_2$. We have to permute the edges, and
the basic ingredient of a permutation is a cyclic permutation. The following lemma shows how to
perturb a realization along a cyclic permutation.

\begin{lemma}\label{lem:single row perturbations}
Let $G_1$ and $G_2$ be two realizations of the same half-regular bipartite degree matrix $\mathcal{M}$. Let $V'$ be a subset of vertices such that for some $u \in U$, each possible colors appears at most once on the edges between $u$ and $V'$ in $G_1$, furthermore, there exists a cyclic permutation $\pi$ on $V'$ such that for all $v \in V'$, the color between $u$ and $v$ in $G_1$ is the color between $u$ and $\pi(v)$ in $G_2$. Then there exists a sequence of colored graphs $G_1 = H_0, H_1, \ldots H_l$ with the following properties
\begin{enumerate}
\item For all $i = 1, \ldots l-1$, $H_i$ is a colored graph satisfying either the properties 1 and 2 in Lemma~\ref{lem:broken-permutation}  or the properties 1 and 2 in Lemma~\ref{lem:broken-permutation-ii}.
\item For all $i = 0, \ldots l-1$, a perturbation exists that transforms $H_i$ into $H_{i+1}$ and perturbs only the edges of two vertices in $U$.
\item $H_l$ is a realization of $\mathcal{M}$ such that for all $v' \in V'$, the color of the edge between $u$ and $v'$ is the color between $u$ and $v'$ in $G_2$, and for all $v \in V \setminus V'$, the color between $u$ and $v$ is the color between $u$ and $v$ in $G_1$.
\end{enumerate} 
\end{lemma}
\begin{proof}
Let $(v_{i_1}, v_{i_2}, \ldots v_{i_r})$ denote the cyclic permutation, and for all $l = 1, 2, \ldots r$, let $c_{j_l}$ be the 
color of the edge between $u$ and $v_{i_l}$ in $G_1$. 
Since $G_2$ is a realization of $\mathcal{M}$, there is a $u'$ such that the color between $u'$
and $v_{i_1}$ is $c_{j_r}$. Indeed, the color between $u$ and $v_{i_1}$ in $G_2$ is $c_{j_r}$ 
(the permutation $\pi$ moves $v_{i_r}$ to $v_{i_1}$), thus, $v_{i_1}$ has an edge with color $c_{j_r}$.
Swap the edges between $u$ and $v_{i_1}$ and between $u'$ and $v_{i_1}$. This will be $H_1$, which satisfies the conditions $1$ and $2$ in Lemma~\ref{lem:broken-permutation}. Indeed, $u$ has a $(+c_{j_r}-c_{j_1})$ deficiency, while $u'$ has a $(+c_{j_1}-c_{j_r})$ deficiency. Clearly, 
$H_0$ and $H_1$ differ only on edges of $u$ and $u'$, furthermore, all edges of $u$ has the 
same color than in $G_1$ except the edge between $u$ and $v_{i_1}$.

Assume that some $H_t$ is achieved for which conditions $1$ and $2$ in Lemma~\ref{lem:broken-permutation} are satisfied with $u$ having $(+c_{j_{r}}-c_{j_s})$ deficiency and with some $u'$ 
having $(+c_{j_s}-c_{j_{r}})$ deficiency and for all vertices $v_{i_1}, v_{i_2}, \ldots v_{i_{s}}$, the edges between $u$ and these vertices have a color as in $G_2$. Furthermore, all other edges between $u$ and $v \ne v_{i_1}, v_{i_2}, \ldots v_{i_s}$ did not change color.
 The color between $u$ and $v_{i_{s+1}}$ in $G_2$ is $c_{j_s}$, therefore, there is a $u"$ such that the color between $u"$
and $v_{i_{s+1}}$ is $c_{j_s}$  in $H_t$. Swap the edges between $u$ and $v_{i_{s+1}}$ and
between $u"$ and $v_{i_{s+1}}$. This will be the graph $H_{t+1}$. Clearly, $H_t$ and 
$H_{t+1}$ differ only on edges of $u$ and $u"$. If $u' = u"$, then $u$ has
$(+c_{j_r}-c_{j_{s+1}})$ deficiency, and $u'$ has $(+c_{j_{s+1}}-c_{j_r})$ deficiency, hence
$H_{t+1}$ satisfies conditions $1$ and $2$ in Lemma~\ref{lem:broken-permutation}. We can rename $H_{t+1}$ to $H_t$ and iterate the chain of transformations.

If $u' \ne u"$, then $u$ has $(+c_{j_r}-c_{j_{s+1}})$ deficiency, $u"$ has 
$(+c_{j_{s+1}}-c_{j_s})$ deficiency and $u'$ has $(+c_{j_s}-c_{j_r})$ deficiency. Thus, the 
conditions $1$ and $2$ in Lemma~\ref{lem:broken-permutation-ii} hold. Due to 
Lemma~\ref{lem:broken-permutation-ii}, there exists a perturbation that affects only edges on
$u'$ and $u"$ and transforms $H_{t+1}$ to an $H_{t+2}$ for which conditions $1$ and $2$ in
Lemma~\ref{lem:broken-permutation} hold with $u$ having $(+c_{j_r}-c_{j_{s+1}})$ deficiency
and $u"$ having $(+c_{j_{s+1}}-c_{j_r})$ deficiency. We can rename $H_{t+2}$ to $H_t$ and iterate the chain of transformations.

With this series of transformations, we can reach $H_t$ wich satisfies conditions $1$ and $2$
in Lemma~\ref{lem:broken-permutation} with $u$ having $(+c_{j_r}-c_{j_{r-1}})$ deficiency
and with some $u'$ having $(+c_{j_{r-1}}-c_{j_r})$ deficiency, furthermore, all vertices 
$v_{i_1}, v_{i_2}, \ldots v_{i_{r-1}}$, the edges between $u$ and these vertices have a
color as in $G_2$, while all other edges of $u$ did not change color.
There is a $u"$ such that the color of the edge between $u"$ and $v_{i_r}$ is $c_{j_{r-1}}$. $H_t$ is transformed into $H_{t+1}$ by swapping the edges between $u$ and $v_{i_r}$ and
between $u"$ and $v_{j_r}$. If $u' = u"$, then this transformation leads to a realization of $\mathcal{M}$, and we are ready, namely, $H_{t+1} = H_l$. Otherwise, $H_{t+1}$ satisfies the
conditions $1$ and $2$ in Lemma~\ref{lem:broken-permutation} with $u'$ having 
$(+c_{j_{r-1}}-c_{j_r})$ deficiency and $u"$ having $(+c_{j_r}-c_{j_{r-1}})$ deficiency. 
According to Lemma~\ref{lem:broken-permutation}, $H_{t+1}$ can be transformed into $H_l$
with a perturbation affecting only edges on $u'$ and $u"$.
\end{proof}

Since $H_l$ is also a realization of $\mathcal{M}$, it is desirable to have a series of transformation 
from $G_1$ to $H_l$ such that all the intermediate graphs are realizations of $\mathcal{M}$. To do this, we need a slightly larger perturbation modifying the edges of three vertices in the regular
vertex set, as stated in the following lemma.
\begin{lemma}\label{lem:one cycle proper perturbations}
Let $G_1$, $G_2$, $V'$ and $\pi$ the same as in Lemma~\ref{lem:single row perturbations}. Then there exists a sequence of colored graphs $G_1 = H'_0, H'_1, \ldots H'_{l'}$ with the following properties
\begin{enumerate}
\item For all $i = 1, \ldots l'$, $H'_i$ is a realization of $\mathcal{M}$.
\item For all $i = 0, \ldots l'-1$, a perturbation exists that transforms $H'_i$ into $H'_{i+1}$ and perturbs only the edges of at most three vertices in $U$.
\item $H'_{l'}$ is a realization of $\mathcal{M}$ such that for all $v' \in V'$, the color of the edge between $u$ and $v'$ is the color between $u$ and $v'$ in $G_2$, and for all $v \in V \setminus V'$, the color between $u$ and $v$ is the color between $u$ and $v$ in $G_1$.
\end{enumerate} 
\end{lemma}
\begin{proof}
Consider the series of colored graphs $H_0, H_1, H_2, \ldots H_l$ obtained in 
Lemma~\ref{lem:single row perturbations}. Find the largest $t_1$ satisfying
$\forall 0 < t' < t_1$, $H_{t'}$ has two vertices with a deficiency. Due to the construction, for all $t'$, these two vertices are the same, and $G_1$ and $H_{t'}$ differ only in edges of these two vertices, $u$ and $u'$. 
If $t_1 = l$, then $G_1$ and $H_l$ differ only on edges of two vertices of $U$, thus $G_1 = H'_0, H'_1 = H_l$ will suffice. Otherwise,
$H_{t_1}$ has deficiency on $3$ vertices, $u$, $u'$ and $u"$.
 $H_{t_1+1}$ has deficiency on $2$ vertices, $u$ and $u"$, and $G_1$ differ from $H_{t_1+1}$ on edges of 
three vertices, $u$, $u'$ and $u"$. Apply Lemma~\ref{lem:broken-permutation} on $H_{t_1+1}$, 
the so-obtained graph will be $H'_1$. $H'_1$ is a realization of $\mathcal{M}$ and differ from $G_1$ only on 
edges of 3 vertices in $U$. 

Iterate this construction, find the largest $t_j$ such that $\forall t_{j-1} < t' < t_j$, $H_{t'}$ has two vertices with deficiency. If $t_j = l$, then $l' = j$, and $H'_{l'} = H_l$. Otherwise, $H_{t_j}$ has 3 vertices with deficiency, and differ from $H'_{j-1}$ only on the edges of these $3$ vertices. $H_{t_j+1}$ has two vertices with deficiency, on which Lemma~\ref{lem:broken-permutation} is
applied to get $H'_j$. $H'_j$ differs from $H'_{j-1}$ only on edges of $3$ vertices, since $H'_{j-1}$ differs from $H_{t_{j-1}+1}$ on the edges of the same 2 vertices that have deficiency for all $H_{t'}$, $t_{j-1} < t' < t_j$. 

After a finite number of iterations, $t_j = l$, and then $j = l'$, $H'_{l'} = H_l$, and this finishes the construction.
\end{proof}

\begin{theorem}\label{theo:full perturbations}
Let $G_1$ and $G_2$ be realizations of the same half-regular bipartite degree matrix $\mathcal{M}$. Then there exists a sequence of colored graphs $G_1 = H_0, H_1, \ldots H_l = G_2$ with the following properties.
\begin{enumerate}
\item For all $i = 0, \ldots l$, $H_i$ is a realization of $\mathcal{M}$.
\item For all $i = 0, \ldots l-1$, a perturbation exists that transforms $H_i$ into $H_{i+1}$ and perturbs only the edges of at most three vertices in $U$.
\end{enumerate}
\end{theorem}
\begin{proof}
Consider a vertex $u$ in the regular vertex class $U$, and define the following directed multigraph
$K$. The vertices of $K$ are the colors $c_1, c_2, \ldots c_k$ and the edges are defined by the 
following way. For each vertex $v \in V$, there is an edge going from $c_i$ to $c_j$, where $c_i$ is the color of the edge between $u$ and $v$ in $G_2$ and $c_j$ is the color of the edge 
between $u$ and $v$ in $G_1$ (if $c_i = c_j$, then this edge will be a loop). Label this edge with vertex $v$.
Since $G_1$ and $G_2$ are realizations of the same degree matrix $\mathcal{M}$, $K$ is Eulerian and thus, can be decomposed into directed cycles. If all cycles are trivial (loops), then for 
all vertex $v \in V$, the color of the edge between $u$ and $v$ in $G_1$ and $G_2$ are the same. If there is a nontrivial cycle $C$, then each color appears in this cycle at most once, and the 
edges of a cycle define a cyclic permutation on a
subset of vertices $V$. Let $(v_{i_1}, v_{i_2}, \ldots v_{i_r})$ denote this cyclic permutation 
$\pi$. By the definition of $K$, the color between $u$ and $v$ in $G_1$ is the color between
$u$ and $\pi(v)$ in $G_2$, for all $v = v_{i_1}, v_{i_2}, \ldots v_{i_r}$. Therefore, we can apply
Lemma~\ref{lem:one cycle proper perturbations} to transform $G_1$ into $H'_{l'}$. Now if we construct the 
same directed multigraph $K$ considering $H'_{l'}$, $G_2$ and the same vertex $u$, then cycle $C$
becomes separated loops for all of its vertices, while other edges are not affected (recall the second half of condition 3 in Lemma~\ref{lem:single row perturbations}, "\emph{for all $v \in V \setminus V'$, the color between $u$ and $v$ is the color between $u$ and $v$ in $G_1$.}"). While there are non-trivial cycles in $K$, we can apply Lemma~\ref{lem:one cycle proper perturbations}, and in a finite
number of steps, $G_1$ is transformed into a realization $H'$ such that for all vertices $v \in V$, the color of the edge between $u$ and $v$ in $H'$ and $G_2$ are the same. Furthermore, in all
steps along the transformation, the intermediate graphs satisfy the two conditions of the lemma.

Fix the edges of $u$ both in $H'$ and $G_2$. Technically, this can be done by deleting vertex $u$ and its corresponding edges both from $H'$ and $G_2$. Then these reduced graphs will be realizations
of the same half-regular degree matrix that can be obtained from $\mathcal{M} = (D,F)$ by deleting a column from the row-regular $D$ (which thus remains row-regular) and modifying
some values of $F$ according to the edge colors of $u$. On these reduced graphs, we can consider a vertex $u$ from the regular vertex class $U$, and do the same. Once there is one remaining vertex in $U$, the two realizations will be the same, and thus, we transformed $G_1$
into $G_2$.
\end{proof}

Theorem~\ref{theo:full perturbations} says that perturbing the edges of at most three vertices in
the regular vertex class is sufficient to connect the space of realizations of a half-regular degree matrix $\mathcal{M}$, namely, a finite series of such perturbations is sufficient to transform any 
realization into another such that all intermediate perturbed graphs are also realizations of $\mathcal{M}$. A natural question is if such perturbations are also necessary. The answer to this question is yes. Latin squares can be considered as 1-factorizations of the complete bipartite graph $K_{n,n}$. The corresponding degree matrix $\mathcal{M} = (D,F)$ satisfies the definition of half-regular degree matrices. Indeed, both $D$ and $F$ are $n\times n$, all-1 matrices. 
It is well-known that already $5 \times 5$ Latin squares cannot be transformed into each other via Latin squares if only two rows are perturbed in a step, and the same claim holds for any $p \times p$ Latin squares, where $p$ is prime and $p\ge 5$ \cite{jm1996}.

\section{Markov Chain Monte Carlo for sampling realizations of a half-regular degree matrix}

In this section, we give a Markov Chain Monte Carlo (MCMC) method for sampling realizations of a half-regular degree matrix $\mathcal{M}$. Theorem~\ref{theo:full perturbations} says that the
perturbations that change the edges of at most 3 vertices of the regular vertex class are irreducible on the realizations of half-regular degree matrices. A na{\"\i}ve approach would make a random perturbation affecting the edges of $3$ vertices, and would accept this random perturbation if it is a realization of $\mathcal{M}$. Unfortunately, the probability that such random perturbation would generate a realization of $\mathcal{M}$  would tend to 0 exponentially quickly with the size of $\mathcal{M}$, making the MCMC approach very inefficient. Indeed, there were an exponential waiting time for an acceptance event, that is, when the random perturbation generates a realization of $\mathcal{M}$. 

Even if the random perturbation generates a realization of $\mathcal{M}$ with probability 1, the Metropolis-Hastings algorithm applying such random perturbation might generate a torpidly mixing Markov chain if the acceptance ratios are small. An example when this is the case can be found in \cite{mms2010}.

In this section, we design a transition kernel that generates such perturbations and its transition probabilities satisfy Equation~\ref{eq:highratio}, namely, the inverse of the acceptance ratio bounded by a polynomial function of the size of $\mathcal{M}$. To do this, we first have to generate random circuits and trails in auxiliary graphs that we define below.

\begin{definition}\label{def:trail}
Let $K(G,u,u')$ be the directed, edge labeled multigraph of the edge colored bipartite graph $G$ as defined in Definition~\ref{def:auxiliaryK}. Let $c_s$ and $c_e$ be two vertices of $K$ such that for any vertex $c\ne c_e$ of $K$, the number of outgoing edges of $c$ is greater or equal than the number of its incoming edges, and $c_s$ has more outgoing edges than incoming edges if $c_s \ne c_e$. We define the following function $f(K,c_s,c_e)$ that generates a random trail from $c_s$ to $c_e$ when $c_s \ne c_e$ and a random circuit when $c_s = c_e$.
\begin{enumerate}
\item Select uniformly a random outgoing edge $e$ of $c_s$ which is not a loop. Let the sequence of labels $P := v$, where $v$ is the label of $e$, namely, $P$ be a sequence containing one label. Let $c$ be the endpoint of $e$.
\item While $c \ne c_e$, select uniformly a random outgoing edge $e$ of $c$ which is not a loop and does not appear in $P$. Extend $P$ with the label of $e$, and set $c$ to the endpoint of $e$.
\end{enumerate}

Note that due to the pigeonhole rule, this random procedure will eventually arrive to $c_e$, since for any vertex $c \ne c_e$ reached by the algorithm, there exist at least one outgoing edge of $c$ which is not a loop and is not in the sequence $P$.
\end{definition}

The probability of a trail $\mathcal{T} = c_s, c_1, c_2, \ldots c_r, c_e$  or circuit 
$\mathcal{C} = (c_s, c_1, c_2, \ldots c_r)$ generated by $f(K, c_s, c_e)$ consists of the product of the inverses of the available edges going out from $c_s, c_1, c_2, \ldots c_r$. These probabilities will be denoted by $P(\mathcal{C})$ and $P(\mathcal{T})$.

We are going to perturb the graphs along the circuits and trails as defined below.

\begin{definition}
Let $\mathcal{C}$ (resp., $\mathcal{T}$) be a random circuit (resp., random trail) generated by $f(K(G,u,u'),c_s, c_e)$. Then the transformation $G*(\mathcal{C},u,u')$ ($G*(\mathcal{T},u,u')$) swaps the colors of the edges between $u$ and $v$ and between $u'$ and $v$ for all edge labels $v$ in $\mathcal{C}$ ($\mathcal{T}$).
\end{definition}

With these types of perturbations, we are going to define the random algorithm that generates a random perturbation.
\begin{figure}
\centering
\includegraphics[trim=7.5cm 4.5cm 6.5cm 6.5cm,angle=0, width=3.2in]{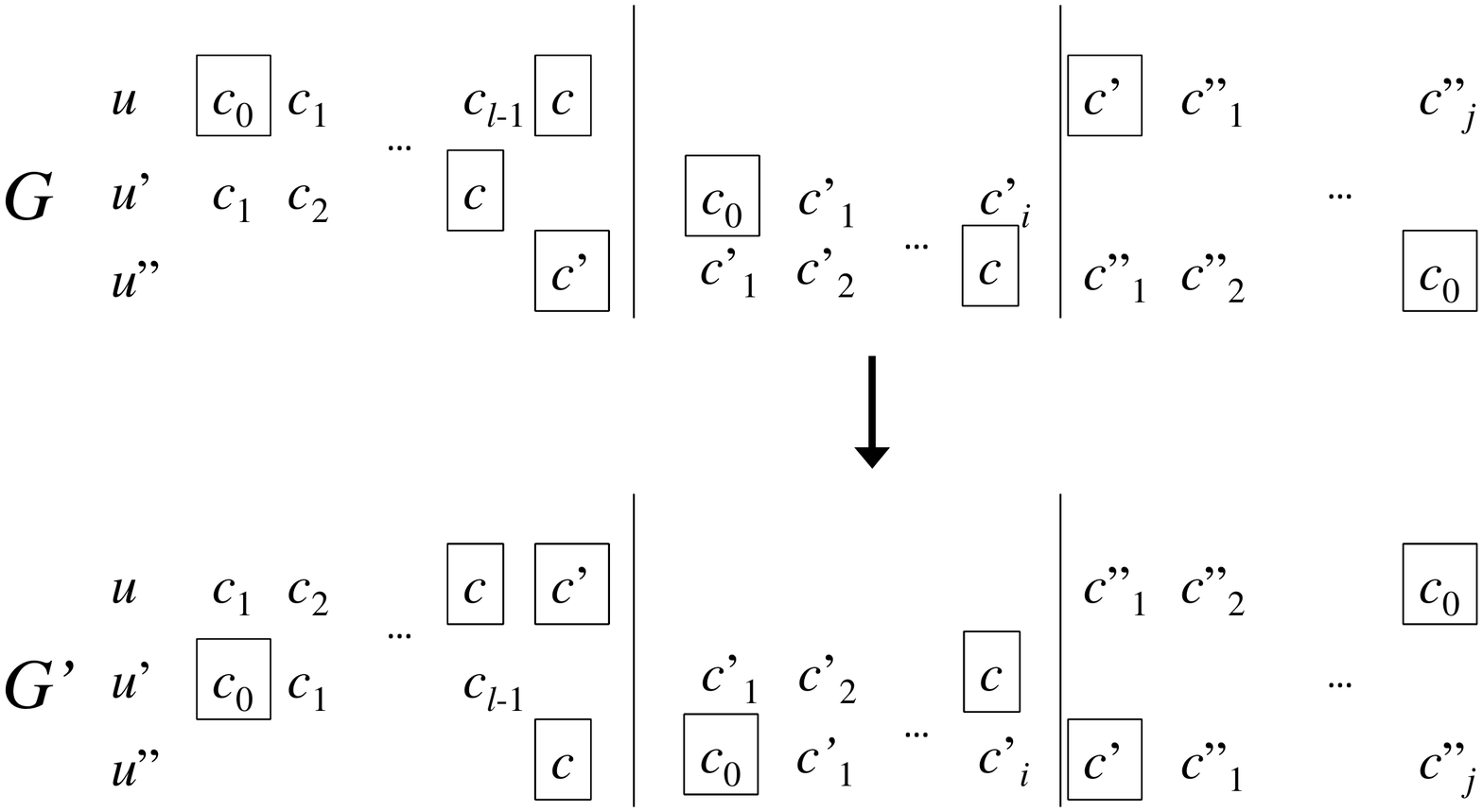}
\caption{This figure shows how Algorithm~\ref{alg:perturbation}, case II.(c)-(e) perturbs $G$ into a $G'$. Vertices $u$, $u'$ and $u"$ are the vertices whose edges are perturbed.
Couples of colors are swapped in three trails (see also Definition~\ref{def:trail}), furthermore a pair of colors $c$ and $c'$ are swapped. Some colors are boxed in order to help understand that the perturbation yields a realization of $\mathcal{M}$, see also the proof of Theorem~\ref{theo:perturbation}.
 The transformation swapping the columns until the first vertical line yields $\tilde{G}$, until the second vertical line yields $\bar{G}$. Note that for sake of readability, all color changes are indicated in separate columns, however, it is allowed that a column is used several times during the perturbation. Also note that due to readability, columns are in order as they follow each other in a trail, however, they might not be necessarily consecutive ones in $G$.}\label{fig:perturbation1}
\end{figure}

\begin{figure}
\centering
\includegraphics[trim=7.5cm 3.5cm 6.5cm 8.5cm,angle=0, width=3.2in]{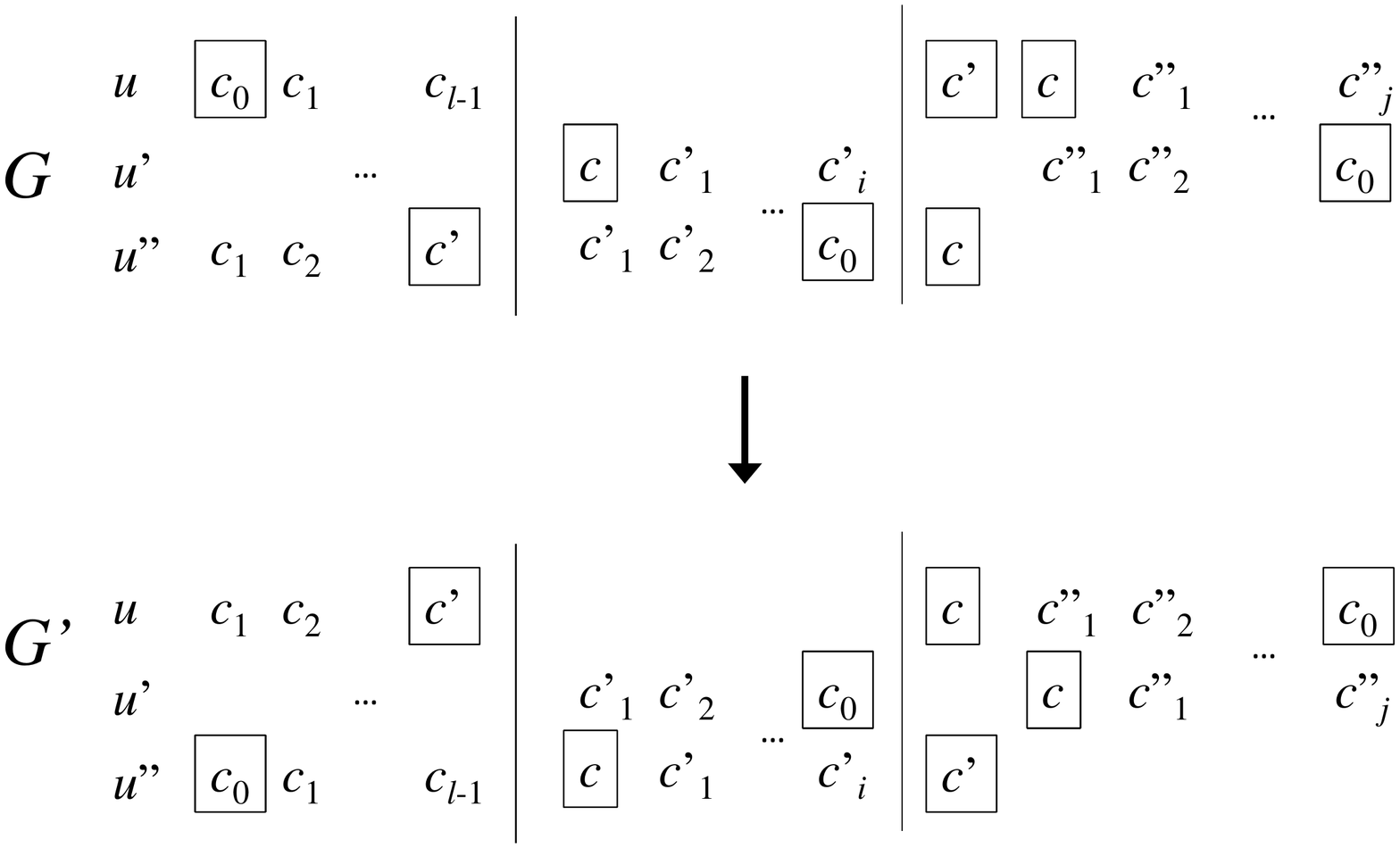}
\caption{This figure shows how Algorithm~\ref{alg:perturbation}, case III.(c)-(e) perturbes $G$ into a $G'$. Vertices $u$, $u'$ and $u"$ are the vertices whose edges are perturbed.
Couples of colors are swapped in three trails (see also Definition~\ref{def:trail}), furthermore a pair of colors $c$ and $c'$ are swapped. Some colors are boxed in order to help understand that the perturbation yields a realization of $\mathcal{M}$, see also the proof of Theorem~\ref{theo:perturbation}.
 The transformation swapping the columns until the first vertical line yields $\bar{G}$, until the second vertical line yields $\tilde{G}$. Note that for sake of readability, all color changes are indicated in separate columns, however, it is allowed that a column is used several times during the perturbation. Also note that due to readability, columns are in order as they follow each other in a trail, however, they might not be necessarily consecutive ones in $G$.}\label{fig:perturbation2}
\end{figure}

\begin{algorithm}\label{alg:perturbation}
Let $G(U,V,E)$ be a realization of the half-regular degree matrix $\mathcal{M}$. Do the following random perturbation on $G$.
\begin{enumerate}[I.]
\item With probability $\frac{1}{2}$, do nothing. This is technically necessary for the Markov chain surely be aperiodic. Any constant probability would suffice, $\frac{1}{2}$ is the standard choice for further technical reasons not detailed here, see for example, \cite{sj1989}.

\item With probability $\frac{1}{4}$, do the following
\begin{enumerate}
\item  Draw uniformly a random ordered pair of vertices $u$ and $u'$ from the vertex set $U$. Construct
the auxiliary graph $K(G, u, u')$, draw uniformly a random color $c_0$ and draw a random circuit $\mathcal{C}$ using $f(K(G,u,u'), c_0, c_0)$.

\item Chose a random integer $l$ uniformly from $[1,m]$. If  $l \ge |\mathcal{C}|$, then let $G'$ be $G*(\mathcal{C},u,u')$. $G'$ is the perturbed graph, exit the algorithm.

\item If $l < |\mathcal{C}|$, do the following (see also Figure~\ref{fig:perturbation1}). Let $e_1, e_2, \ldots$ denote the edges of circuit $\mathcal{C}$ starting with the edge going out from $c_0$.  Let $\mathcal{T}$ be the trail with edges $e_1, e_2, \ldots e_l$. 
Perturb $G$ to $G*(\mathcal{T},u,u')$.
Draw uniformly a $u"$ from $U \setminus \{u, u'\}$ and also uniformly a vertex $v$ from the subset of $V$ satisfying the condition that the color of the edge between $u$ and $v$  in $G*(\mathcal{T},u,u')$ is the last color in the trail $\mathcal{T}$. Let $c'$ denote the color of the edge between $u"$ and $v$ . Swap the colors of the edges between $u$ and $v$ and between $u"$ and $v$. Note that $c'$ might equal to $c$, and in that case, swapping the colors has no effect. Denote the so far perturbed graph $\tilde{G}$.

\item Construct $K(\tilde{G},u',u")$, and draw a random trail $\mathcal{T}_1$ using function $f(K(\tilde{G},u',u"), c_0, c)$.
Let $\bar{G}$ be $\tilde{G}*(\mathcal{T}_1, u',u")$.

\item Construct $K(\bar{G},u, u")$, and draw a random trail $\mathcal{T}_2$ using function $f(K(\bar{G},u,u"), c', c_0)$. Let $G'$, the perturbed graph be $\bar{G}*(\mathcal{T}_2,u,u")$.
\end{enumerate}

\item With probability $\frac{1}{4}$, do the following.
\begin{enumerate}
\item Draw uniformly a random ordered pair of vertices $u$ and $u"$ from the vertex set $U$. Construct $K(G,u,u")$, draw uniformly a random color $c_0$, and draw a random circuit $\mathcal{C'} = (c_0, c_1, \ldots)$ using $f(K(G,u,u"), c_0, c_0)$.

\item Draw uniformly a random integer from $[1,m]$. If  $l \ge |\mathcal{C'}|$, then let $G'$ be $G*(\mathcal{C'},u,u")$. $G'$ is the perturbed graph, exit the algorithm. 

\item If $l < |\mathcal{C}'|$, do the following (see also Figure~\ref{fig:perturbation2}). Let $\mathcal{T'}$ be the trail from $c_0$ to $c_l$, and let $\bar{G} = G*(\mathcal{T'},u,u")$, and let $c'$ denote $c_l$.
Draw uniformly a random vertex $u'$ from $U \setminus \{u, u"\}$. Build $K(\bar{G},u",u')$, and draw a random trail $\mathcal{T}_1'$ using $f(K(\bar{G},u",u'),c_0,c')$. Let $\mathcal{T}_1'[i]$ denote the prefix of the trail $\mathcal{T}_1'$ containing only the first $i$ edges of the trail. Let $e_i$ denote the $i$th edge of trail $\mathcal{T}_1'$. Construct the set of vertices $V'$ which contains all vertex $v_i$ such that the label of $e_i$ is $v_i$ and in the graph $\bar{G}*(\mathcal{T}_1'[i], u", u')$, there is a vertex $v$ such that the color of the edge between $u$ and $v$ is $c'$, the color of the edge between $u"$ and $v$ is the color of the edge between $u'$ and $v_i$, furthermore, $c$ is the first occurrence in the trail, that is, $\mathcal{T}'_1[i-1]$ does not contain $c$. 

\item 
If $V'$ is empty, then exit the algorithm, let the perturbed graph be the original graph. Otherwise,
draw a random vertex $v$ from $V'$. Shorten $\mathcal{T}_1'$ such that the last edge have label $v$. Let $\mathcal{T}_1"$ denote this trail. Let $\tilde{G}$ be $\bar{G}*(\mathcal{T}_1", u", u')$. If $u"$ has a deficiency $(+c-c')$ for some $c$ in $\tilde{G}$, then draw uniformly a vertex among the vertices $v"$ for which the color of the edge between $u$ and $v"$ is $c'$ and the color of the edge between $u"$ and $v"$ is $c$. Modify $\tilde{G}$ by swapping the colors of these two edges.

If $u"$ has no deficiency in $\tilde{G}$ then recall $c'$ to $c$.

\item Build $K(\tilde{G},u,u')$, and draw a random trail $\mathcal{T}_2'$ using $f(K(\tilde{G},u,u'), c, c_0)$. Let $G'$, the perturbed graph be $\tilde{G}*(\mathcal{T}_2',u,u')$.

\end{enumerate}

\end{enumerate}

\end{algorithm}

\begin{theorem}\label{theo:perturbation}
The random perturbation in Algorithm~\ref{alg:perturbation} has the following properties.
\begin{enumerate}
\item The generated $G'$ is a realization of $\mathcal{M}$, thus this random perturbation is the transition kernel of a Markov chain on the realizations of $\mathcal{M}$.
\item The perturbations are irreducible on the realizations of $\mathcal{M}$
\item The perturbations form a reversible kernel and for any $G$ and any $G'$ generated from $G$, 
\begin{equation}
\frac{T(G|G')}{T(G'|G)} \ge \frac{2}{m^5} \label{eq:claim}
\end{equation}
where $m$ is the number of vertices in $V$. Furthermore, if this reversible kernel is used in a Metropolis-Hastings algorithm, the expected waiting time (that is, how many Markov chain steps is necessary to leave the current state) at any state is bounded by $2m^5$.
\end{enumerate}
\end{theorem}
\begin{proof}
\hfill
\begin{enumerate}
\item When $l \ge |\mathcal{C}|$ or $l \ge |\mathcal{C}'|$ (cases II.(b) and III.(b) in Algorithm~\ref{alg:perturbation}), the constructed $G'$ differs on edges of two vertices of $U$, $u$ and $u'$ (or $u$ and $u"$). Since the swapped colors are along a circuit in the auxiliary graph $K$, the effects of color swaps cancel each other, and we get a realization of $\mathcal{M}$. 

When $l < |\mathcal{C}|$ (case II.(c)-(e) in Algorithm~\ref{alg:perturbation}), $G$ is transformed into an intermediate graph $\tilde{G}$, which has deficiency on three vertices $u$, $u'$ and $u"$. 
Vertex $u$ has $(+c'-c_0)$-deficiency, $u'$ has $(+c_0-c)$-deficiency, and $u"$ has $(+c-c')$-deficiency (see also the boxed colors on Figure~\ref{fig:perturbation1}).
The generated random trail $\mathcal{T}_1$ on $K(\tilde{G},u'u")$ goes from $c_0$ to $c$. In Lemma~\ref{lem:broken-permutation-ii}, we proved that any such trail transforms $\tilde{G}$ into a graph satisfying the conditions of Lemma~\ref{lem:broken-permutation}. The random trail $\mathcal{T}_2$ generated in $K(\bar{G},u,u")$ is from $c'$ to $c_0$, and Lemma~\ref{lem:broken-permutation} proves that swapping the edge colors along such trail transforms the graph into a realization of $\mathcal{M}$.

When $l < |\mathcal{C}'|$ (case III.(c)-(e) in Algorithm~\ref{alg:perturbation}), $G$ is transformed into a $\bar{G}$, in which $u$ has a $(+c'-c_0)$-deficiency and $u"$ has a $(+c_0-c')$ deficiency (see also the boxed colors on Figure~\ref{fig:perturbation2}). Then $\bar{G}$ transformed into a $\tilde{G}$, in which two cases is possible. Either $u"$ has a $(+c-c')$-deficiency and then $u$ has a $(+c'-c_0)$-deficiency and $u'$ has a $(+c_0-c)$-deficiency or $u"$ has no deficiency and then $u$ has a $(+c'-c_0)$-deficiency and $u'$ has a $(+c_0-c')$-deficiency. In the former case, a couple of colors $c$ and $c'$ are swapped between $u$ and $u"$ causing $u"$ has no deficiency and $u$ has $(+c-c_0)$ deficiency, which are cancelled by the perturbation along the trail $\mathcal{T}_2'$, thus arriving to a realization of $\mathcal{M}$. In the later case, the perturbation along the trail $\mathcal{T}_2'$ directly leads to a realization of $\mathcal{M}$.

\item It is sufficient to show that the algorithm generates the perturbations appear in Theorem~\ref{theo:full perturbations}, which are actually the perturbations appear in Lemma~\ref{lem:one cycle proper perturbations}. There are two types of perturbations in Lemma~\ref{lem:one cycle proper perturbations}. The simpler one swaps the colors of edges along a cycle. This happens when $t_1 = l$ in the proof of Lemma~\ref{lem:one cycle proper perturbations} and also when $t_j = l$. Since a cycle is a circuit, and any circuit in the auxiliary graph $K$ constructed in the algorithm can be generated with non-zero probability, this type of perturbation is in the transition kernel.

The larger perturbation in Lemma~\ref{lem:one cycle proper perturbations} first swap the colors of edges of $u$ and $u'$ along a path, then swap the colors of two edges between $u$ and $v$ and $u"$ and $v$ for some $v$, then swap the edges of $u'$ and $u"$ along a path to eliminate the deficiency of $u"$, finally, swap the color of the edges $u$ and $u"$
along a path to eliminate their deficiencies, thus generate a realization. Since any path is a trail, the given algorithm can generate such perturbations when $l < |\mathcal{C}|$ (case II.(c)-(e) in Algorithm~\ref{alg:perturbation}).

\item We show that for any perturbation of $G$ to $G'$, there exists a perturbation from $G'$ to $G$, and we also compare the probabilities of the perturbations. Recall that the number of vertices in $U$ is $n$, the number of vertices in $V$ is $m$, and the number of colors is $k$; $n$, $m$ and $k$ will appear in the probabilities below several times. 

When the perturbation affects the edges of only two vertices, $u$ and $u'$ or $u$ and $u"$ (cases II.(b) or III.(b) in Algorithm~\ref{alg:perturbation}), the algorithm first draws these vertices with probability
$\frac{1}{n(n-1)}$, where $n = |U|$. Also a random color $c_0$ is drawn with probability $\frac{1}{k}$. Then a random circuit $\mathcal{C} = (c_0, c_1, c_2, \ldots c_r)$ is generated with probability $P(\mathcal{C})$. Finally, a random $l \ge r$ is generated, this happens with $\frac{m-r-1}{m}$ probability, where $r = |\mathcal{C}|$ and the color of the edges indicated by the circuit are swapped. The so-obtained $G'$ can be generated not only in this way, but an arbitrary such $c_i$ can be generated as the starting point of the circuit which is visited only once in the circuit.

To perturb back $G'$ to $G$, we first have to select the same pair of vertices $u$ and $u'$, with the same, $\frac{1}{n(n-1)}$ probability. Then the random circuit $\mathcal{C}' = (c_0, c_r, c_{r-1}, \ldots c_1)$ must be generated in the constructed auxiliary directed multigraph $K'$. Since for all color, the auxiliary graph $K$ constructed from $G$ has the same outgoing edges than the auxiliary graph $K'$ constructed from $G'$, $P(\mathcal{C}) = P(\mathcal{C}')$. Furthermore, this claim holds for all possible circuits having the possible starting colors $c_i$. Finally, generating the same $l$ has the same probability, thus we can conclude that for this type of perturbation, $P(G'|G) = P(G|G')$, thus Equation~\ref{eq:claim} holds.

When the perturbation affects the edges of three vertices, $u$, $u'$ and $u"$, then $G$ is first perturbed to $\tilde{G}$, $\tilde{G}$ is perturbed to $\bar{G}$ and finally $\bar{G}$ is perturbed to $G'$, or the third type of perturbation is applied in the algorithm where $G$ first perturbed into $\bar{G}$ then $\tilde{G}$ and finally $G'$ (case II.(c)-(e) and III.(c)-(e) in Algorithm~\ref{alg:perturbation}). We show that these two types of perturbations are inverses of each other in the following sense. For any perturbation generated in case II.(c)-(e) in Algorithm~\ref{alg:perturbation}, its inverse can be generated in case III.(c)-(e) in Algorithm~\ref{alg:perturbation}, and vice versa, for any perturbation generated in case III.(c)-(e), its inverse can be generated in case II.(c)-(e).

First we consider the perturbation generated by Algorithm~\ref{alg:perturbation} in case II.(c)-(e). In that case, first the ordered pair of vertices ($u$, $u'$) is generated with probability $\frac{1}{n(n-1)}$ and a random color $c_0$ with probability $\frac{1}{k}$. A circuit $\mathcal{\mathcal{C}} = (c_0, c_1, \ldots c_{l}, \ldots c_r)$ is generated together with a random number $l$ with probability $\frac{1}{m}$. The generated number $l$ must be smaller than $|\mathcal{C}|$, and thus, only the trail $\mathcal{T} = c_0, c_1, \ldots c_l$ is important. The probability of the trail consists of the product of the inverses of the number of available outgoing edges. Let $P(\mathcal{T})$ denote this probability.  The colors along the trail are swapped. Then random $u"$ is generated from the appropriate set of vertices in $U\setminus \{u,u'\}$ with probability $\frac{1}{n-2}$
and also a vertex $v$ is generated with probability $\frac{1}{d_{j,1}+1}$ where $j$ is the index of color $c_l$. The colors $c (= c_l)$ and $c'$ are swapped. Therefore, $P(\tilde{G}|G) = \frac{1}{n(n-1)(n-2)km(d_{j,1}+1)}P(\mathcal{T})$.
Then a trail $\mathcal{T}_1$ is generated from $c'$ to $c_0$ in the auxiliary graph $K(\tilde{G},u'u")$, and $\tilde{G}$ is transformed to $\bar{G} = \tilde{G}*(\mathcal{T}_1,u',u")$ Thus, $P(\bar{G}|\tilde{G}) = P(\mathcal{T}_1)$.
 Finally, a random trail $\mathcal{T}_2$ from $c'$ to $c_0$ is generated in $K(\bar{G},u,u")$ with probability $P(\mathcal{T}_2)$, which is $P(G'|\bar{G})$. Since
$$P(G'|G) = P(\tilde{G}|G)P(\bar{G}|\tilde{G})P(G'|\bar{G})$$
we get that
\begin{equation}
P(G'|G) = \frac{1}{n(n-1)(n-2)km(d_{j,1}+1)}P(\mathcal{T})P(\mathcal{T}_1)P(\mathcal{T}_2). \label{eq:G2Gprime}
\end{equation}

To transform back $G'$ to $G$, first $G'$ should be transformed back to $\bar{G}$, then $\bar{G}$ back to $\tilde{G}$ and finally $\tilde{G}$ back to $G$. This can be done in case III.(c)-(e) by first drawing the same $u$ and $u"$, then drawing a $\mathcal{T}'$ which is exactly the inverse of $\mathcal{T}_2$, then the same $u'$ must be selected and the trail $\mathcal{T}_1"$ must be the inverse of $\mathcal{T}_1$, the same edges with color $c$ and $c'$ must be swapped, and finally the trail $\mathcal{T}_2'$ must be the inverse of $\mathcal{T}$. First we show that these trails can be inverses of each other.

$\mathcal{T}$ is a shortening of a circuit with start and end vertex $c_0$. As such, it contains $c_0$ only once, but might contain $c$ several times. $\mathcal{T}'_2$ is a trail from $c$ to $c_0$. Therefore, it can contain $c$ several times, but contains $c_0$ only once. Thus, the inverse of any $\mathcal{T}$ might be a $\mathcal{T}'_2$ and vice versa.

$\mathcal{T}_1$ is a trail from $c_0$ to $c$. Therefore, it might contain $c_0$ several times, but contains $c$ only once. $\mathcal{T}"_1$ is a shortening of a trail from $c_0$ to $c'$. It might contain $c_0$ several times, but can contain $c$ only once, due to its definition (see case III.(c)-(d) of Algorithm~\ref{alg:perturbation}). Hence, the inverse of any $\mathcal{T}_1$ can be a $\mathcal{T}"_1$, and vice versa.

$\mathcal{T}_2$ is a trail from color $c'$ to $c_0$. It might hit $c'$ several times, but only once $c_0$. $\mathcal{T}'$ is a shortening of a circuit with start and end vertex $c_0$. As such, it contains $c_0$ only once (as start and end vertex), but might contain $c'$ several times. Thus the inverse of any $\mathcal{T}_2$ might be a $\mathcal{T}'$ and vice versa.

We are going to calculate the probability of a random perturbation in case III.(c)-(e). First, the ordered pair of vertices $u$ and $u"$ should be selected with probability $\frac{1}{n(n-1)}$. The auxiliary directed multigraph $K(G',u,u")$ is constructed, a random $c_0$ is selected, and a random circuit $(c_0, c_1 \ldots c', \ldots c_{r'})$ in $K(G',u,u")$ is generated. The probability that the randomly generated $l$ is exactly the index of the appropriate occurrence of $c'$ in the trail $\mathcal{T}'$ is $\frac{1}{m}$. Then only the trail $\mathcal{T}' = c_0, c_1, \ldots c'$ is interesting. Thus, transforming back $G'$ to $\bar{G}$ has probability $\frac{1}{n(n-1)km}P(\mathcal{T}')$. $\mathcal{T}'$ in $K(G',u,u")$ is the inverse of the trail $\mathcal{T}_2$ in $K(\bar{G},u,u")$ and $K(G',u,u")$ and $K(\bar{G},u,u")$ differs in inverting the trail $\mathcal{T}_2$. Therefore $P(\mathcal{T}_2)$ and $P(\mathcal{T}')$ differ in the number of outgoing edges from $c_0$ at the begining of the trail $\mathcal{T}_2$ and the number of outgoing edges from $c'$ at the beginning of the trail $\mathcal{T}'$. Since the number of outgoing edges might vary between 1 and $m$, the ratio of the two probabilities bounded by
\begin{equation}
\frac{1}{m} \le \frac{P(\mathcal{T}')}{P(\mathcal{T}_2)} \le m \label{eq:bound}
\end{equation}

 After swapping the colors along the trail $T'$, $u'$ should be randomly generated, it has probability $\frac{1}{n-2}$.
A random trail $\mathcal{T}_1'$ is generated using $f(K(\bar{G}, u', u"), c_0, c')$. A random $v$ from the subset $V'$ is generated in III.(d) of the algorithm. This has probability $\frac{1}{|V'|}$. Then the trail $\mathcal{T}_1'$ is shortened to $\mathcal{T}_1"$, and $\tilde{G}$ is obtained by transforming $\bar{G}$ along this trail. 

It is not easy to calculate exactly the probability $P(\tilde{G}|\bar{G})$ since the set $V'$ depends on the generated trail $\mathcal{T}_1'$. However, the size of $V'$ cannot be greater than $m$ and lower than $1$, therefore the following inequality holds:

$$\frac{1}{m}P(\mathcal{T}_1") \le P(\tilde{G}|\bar{G}) \le P(\mathcal{T}_1")$$
The trail $\mathcal{T}_1"$ is the inverse trail $\mathcal{T}_1$, therefore, the ratio of their probabilities is also between $\frac{1}{m}$ and $m$.

Finally, the same edge colors $c$ and $c'$ must be swapped back, and $\tilde{G}$ must be transformed back  to $G$ along the trail $\mathcal{T}_2'$. The probability that the selected $v"$ in III.(d) of the Algorithm~\ref{alg:perturbation} is a particular vertex depends on the number of vertices from which $v"$ is selected. Therefore this probability is again between $\frac{1}{m}$ and $1$. Altogether, the probability of the backproposal probability is bounded between
\begin{eqnarray}
\frac{1}{n(n-1)(n-2)km^3}P(\mathcal{T}')P(\mathcal{T}_1")P(\mathcal{T}_2') & \le P(G|G')  \le & \nonumber \\
\le \frac{1}{n(n-1)(n-2)km}P(\mathcal{T}')P(\mathcal{T}_1")P(\mathcal{T}_2') &&
\label{eq:Gprime2G}
\end{eqnarray} 
 
Comparing Equations~\ref{eq:G2Gprime}~and~\ref{eq:Gprime2G}, and also considering that the ratio of the probabilities of the trails and corresponding inverse trails are between $\frac{1}{m}$ and $m$, we get for the ratio of proposal and backproposal probabilities that

\begin{equation}
\frac{2}{m^5} \le
\frac{P(G|G')}{P(G'|G)} \le m^4 \label{eq:ratio}
\end{equation}

Case II. is chosen with probability $\frac{1}{4}$. Given that case II. is selected, the probability that Algorithm~\ref{alg:perturbation} generates a realization being different from the current realization is 1. If edges of two vertices in $U$ are perturbed, then the perturbation is accepted with probability 1. If edges of three vertices in $U$ are perturbed, then the perturbation is accepted with probability
$$\frac{P(G|G')}{P(G'|G)} \ge \frac{2}{m^5}$$
according to Equation~\ref{eq:ratio}. Thus, the probability that the Markov chain defined by the Metropolis-Hastings algorithm does not remain in the same state is greater or equal than $\frac{1}{2m^5}$. The expected waiting time is upper bounded by the inverse of this probability, that is, $2m^5$.
\end{enumerate}
\end{proof}

\section{Concluding remarks and future works}

We considered the half-regular factorizations of the complete bipartite graph, and we proved that it is a tractable version of the edge packing problem. Above the existence theorem, we also give sufficient and necessary perturbations to transform solutions into each other. When these perturbations are the transition kernel of a Markov chain Monte Carlo method, the inverse of the acceptance ratios are polynomial bounded. This result might be the first step to prove rapid mixing of the Markov chain. However, proving rapid mixing might be particularly hard. The speed of convergence of a similar Markov chain on Latin squares is a twenty years old open problem.

Approximate sampling and counting are equally hard computational problems for a large set of counting problems, the so-called self-reducible counting problems \cite{jvv1986}. Latin squares are not self-reducible counting problems, since deleting a few rows from a Latin square yield a partial Latin square. Half-regular factorizations of the complete bipartite graph are closer to be self-reducible, since deleting a few lines from a half-regular factorization yields another half-regular factorization of a smaller complete bipartite graph. Unfortunately, further technical conditions are necessary; we have to require that a subset of the colors in the first line be fixed. It is discussed in \cite{ekms2015}, why this restriction is necessary. A possible further work could be to give a Markov chain which is irreducible on such restricted space. A proof of rapid mixing of that Markov chain would yield to an efficient random approximation (FPRAS, see also \cite{vazirani2003}) on the number of realizations of half-regular factorizations of the complete bipartite graph. Since Latin squares are also half-regular factorizations, it would also yield an efficient random approximation of Latin squares, too. At the moment, the known lower and upper bounds on the number Latin squares are far away each other \cite{vlw1992}.

Finally, a further work might be to relax the half-regularity to a condition where we require that the factors be almost half-regular, that is, for each factor, the degrees in the first vertex class are either $d_i$ or $d_i+1$.

\bibliographystyle{plain}

\end{document}